\documentclass[11pt]{amsart}
\usepackage{geometry}                
\usepackage{graphicx}
\usepackage{amssymb}
\usepackage{epstopdf}
\usepackage{mathtools}

\usepackage{multirow}

\usepackage{sagetex}

\usepackage{tikz}

\newcommand{\Q}{\mathbb{Q}}                 
\newcommand{\N}{\Bbb N}                 
\newcommand{\Z}{\Bbb Z}                 
\newcommand{\R}{\Bbb R}                 
\newcommand{\C}{\Bbb C}                 
\newcommand{\F}{\Bbb F}                  

\newcommand{\x}{\vec{x}}
\newcommand{\y}{\vec{y}}

\renewcommand{\(}{\left(}
\renewcommand{\)}{\right)}

\newcommand{\Cls}{\mathrm{Cls}}
\newcommand{\Gen}{\mathrm{Gen}}

\newcommand{\Aut}{\mathrm{Aut}}

\newcommand{\Mass}{\mathrm{Mass}}

\newcommand{\Set}{\Bbb S}                  
\newcommand{\al}{\alpha}
\newcommand{\ra}{\rightarrow}

\newcommand{\ve}{\varepsilon}

\newtheorem{thm}{Theorem}[section]

\newtheorem{lem}[thm]{Lemma}

\newtheorem{defn}[thm]{Definition}

\theoremstyle{remark}
\newtheorem{rem}[thm]{Remark}

\newcommand{\ord}{\text{ord}}

\newcommand{\Hyp}{\mathbb{H}}
\newcommand{\T}{\mathbb{T}}
\newcommand{\leg}[2]{\(\frac{#1}{#2}\)} 
\newcommand{\Cond}{\text{Cond}}

\title{Enumerating maximal definite quadratic forms \\ of bounded class number over $\Z$ in $n\geq 3$ variables}
\author{Jonathan Hanke}

\begin{document}
\maketitle

\begin{abstract}
In this paper we give an algorithm for enumerating all primitive (positive) definite maximal $\Z$-valued quadratic 
forms  $Q$ 
in $n\geq 3$ variables 
with bounded class number $h(Q)\leq B$.
We do this by analyzing the exact mass formula \cite{GHY}, and bounding all relevant local invariants to give only finitely many possibilities.
We also briefly describe an open-source implementation of this algorithm we have written in Python/Sage \cite{sage-4.6.2}  which explicitly enumerates all such 
quadratic forms of bounded class number in $n\geq 3$ variables. 
Using this we determine that there are exactly 115 primitive positive definite maximal $\Z$-valued quadratic forms in $n\geq 3$ variables of class number one, and produce a list of them.

%
In a future paper  we will complete this chain of ideas by extending these algorithms to allow the enumeration of all primitive maximal totally definite $O_F$-valued quadratic lattices of rank $n \geq 3$, where $O_F$ is the ring of integers of any totally real number field $F$.
\end{abstract}

\section{Brief History and Context}\label{Sec:Brief_History}


Given a positive definite $\Z$-valued quadratic form $Q(\x)$ in $n$ variables, one is often interested in understanding which numbers $m$ are {\bf represented} in the form $Q(\x) = m$ for some $\x \in \Z^n$.  The numbers $m$ represented by $Q$ are invariant under any invertible linear change of variables of $Q$ with coefficients in $\Z$, and we say that any two quadratic forms that are equivalent in this way are in the same {\bf class} $\Cls(Q)$.  Similarly, we can define the {\bf genus} $\Gen(Q)$ as the set of quadratic forms that are equivalent to $Q$ (under invertible linear change of variables) over $\R$ and over $\Z/M\Z$ for every $M \in \N$.   
Due to the relatively simple nature of quadratic forms over $\R$ and $\Z/M\Z$, an important measure of the complexity of a quadratic form $Q$ is given by its {\bf class number} $h(Q)$, which is defined to be the number of classes in its genus.  An important result of Siegel shows that this class number $h(Q)$ is always finite.


\smallskip

The question of enumerating quadratic forms of small class number goes back to Euler who around 1778 used positive definite binary quadratic forms (i.e. $n=2$) of class number one for primality testing.  (The determinants of such forms are called {\bf idoneal numbers}.)   
From the perspective of algebraic number theory, it is more natural to consider the number of {\bf proper classes} $h(\Delta_Q)$ of binary quadratic forms $Q(x,y) := ax^2 + bxy + cy^2$ of a given discriminant $\Delta_Q := b^2 - 4ac$, where we say that two quadratic forms are in the same proper class (denoted $\Cls^+(Q)$) if they are are equivalent by some linear change of variables in $\Z$ of determinant one.  The connection between binary quadratic forms and quadratic fields was given by Dedekind (see \cite[Ch XIII, \textsection 5]{Cohn1980}),
%
%
who showed that
%
there is a bijective correspondence between ideal classes of the imaginary quadratic field $\Q(\sqrt{\Delta})$ and the proper classes of primitive positive definite binary quadratic forms $Q(x,y)$ of (fundamental) discriminant $\Delta_Q = \Delta < 0$.  We denote this common {\bf (field) class number} by $h(\Delta)$.  
This correspondence also can be extended to quadratic orders of non-fundamental discriminant.
%
%
%
In the language of quadratic forms, it was conjectured by Gauss \cite{Gauss:1986uq} that for imaginary quadratic fields 
the class number $h(\Delta) \ra \infty$ as $\Delta \ra -\infty$.
This conjecture was proved by Siegel \cite{Siegel:1935fk} in 1935, though the lower bound for $h(\Delta)$ given there is ineffective due to the possibility of a ``Siegel zero'' for some quadratic Dirichlet $L$-function.  Only much later, in the 1980's, 
was an effective algorithm was established for provably enumerating all fundamental discriminants $\Delta < 0$ with 
bounded class number $h(\Delta) \leq B$
%
by using weaker effective lower bounds coming from the $L$-function of a rank $\geq3$ elliptic curve (see \cite{Goldfeld:1985zr}).  Unfortunately this idea is not enough to give an effective lower bound on the possibly smaller class numbers $h(Q)$ of definite binary quadratic forms (counting classes within a single genus $\Gen(Q)$), and this question is presently unresolved.

%

\smallskip

Aside from their historical significance, quadratic forms of small class number also play an important role in understanding what numbers $m$ are represented by a given quadratic form.  It is known that if $m$ is represented by $Q$ over $\R$ and over $\Z/M\Z$ for all $M$ (i.e. $m$ is {\bf locally represented} by $Q$), then $m$ is represented by some form $Q' \in \Gen(Q)$.  When $h(Q)=1$ this means that $m$ is represented by $Q$, so we know exactly what numbers $Q$ represents.  By using the Siegel-Weil formula when $h(Q)=1$ one can also give a quantitative version of this statement that produces simple formulas for the exact number of representations of $m$ in terms of divisor sums and class numbers of quadratic fields (see e.g. \cite[\textsection2.5]{Hanke:AWS2009}).  The prototypical example of this kind of formula is Jacobi's theorem for the sum of 4 squares, which states that for $m \in \Z >0$ we have the exact formula 
$$
r_4(m) = 8 \sum_{0 < d \mid m, 4\nmid d} d
$$
for the number of ways $r_4(m)$ of representing $m$ as $x^2 + y^2 + z^2 + w^2$.

The enumeration of forms of small class number is particularly important for understanding representation questions for definite ternary quadratic forms (i.e. $n=3$), since here there is no unconditional effective analytic method that allows one to understand what numbers are represented.  The list of positive definite  ternary forms with $h(Q)=1$, and a generalization of them called ``regular'' forms \cite{Jagy:1997aa}, were used by Bhargava to prove Conway's 15-Theorem \cite{Bhargava:2000kx}.  Also \cite{Ono:1997nz} used the fact that Ramanujan's ternary quadratic form $Q(\x) = x^2 + y^2 + 10z^2$ has class number $h(Q) = 2$ to give an effective lower bound on the numbers that it fails to represent over $\Z$ but represents locally, conditional upon several versions of the Riemann hypothesis.

\smallskip

There have been several efforts to understand and describe the definite quadratic forms of small class number, especially those where $h(Q)=1$, by various approaches.  
By analytic techniques, in 1937 Magnus \cite{Magnus:1937ly} showed that $h(Q) = 1 \implies n < 35$.   By similar analyses Pfeuffer \cite{Pfeuffer:1979bh, Pfeuffer:1971pd} was able to show that there are only finitely many $O_F$-valued totally definite quadratic forms $Q$ over totally real number fields $F$ with bounded class number $h(Q)$.

The most comprehensive results about forms of small class number were given in a series of papers by Watson \cite{Watson:1962aa}--\cite{Watson:1984aa}, where he shows that $h(Q) = 1 \implies n \leq 10$ and then proceeds with an elaborate case-by-case analysis to enumerate many such forms in each dimension by  geometry of numbers style arguments (i.e. analyzing the numbers represented by these forms, and relating these numbers to discriminants).  While he did not finish his enumeration before passing away, he was able to enumerate all class number one forms when $n=3, 7, 8, 9, 10$, and under additional assumptions he also has results for $4\leq n \leq 6$.
%

By another method, studying indecomposable lattices, Gerstein \cite{Gerstein:1972ya, Gerstein:1973wb} was able to construct distinct classes of lattices in a genus and gave the bound $h(Q) \geq p(t)$ (where $p(t)$ is the partition function) if $n \geq 16t+5$ and $t \in \N$.  He also showed that $h(Q) = 1  \implies n \leq 10$ for totally definite forms over totally real number fields.
%
%



In the case of maximal quadratic forms (which we treat in this paper), it was recently determined by Shimura \cite[Thrm 6.4, p361]{Shimura:2006ac}, using quaternionic trace formula computations of \cite{Pizer:1973vn}, that there are 64 classes of primitive positive definite maximal $\Z$-valued quadratic forms in $n=3$ variables.

%

%




%


\section{Results and Strategy} \label{Sec:Strategy}

In this paper we study the class number $h(Q)$ of a definite quadratic form $Q$ by studying a related quantity $\Mass(Q)$ called the {\bf mass} of $Q$.  The mass is a rational number depending only on the genus $\Gen(Q)$, defined as
$$
\Mass(Q):= \sum_{\Cls(Q') \in \Gen(Q)}\frac{1}{|\Aut(Q')|}.
$$
As a consequence of knowing the Tamagawa number of the special orthogonal group $\tau(O^+(Q)) = 2$, the mass also admits a general description as an infinite product 
$$
\Mass(Q) = 2 \prod_v \beta_v(Q)^{-1}
$$
of {\bf local densities} $\beta_v(Q)$ over all places $v$ of $\Q$, 
though the individual local densities may be difficult to compute at primes $p \mid 2\det(Q)$. 
The mass is related to the class number by the formula
$$
\Mass(Q):= \sum_{\Cls(Q') \in \Gen(Q)}\frac{1}{|\Aut(Q')|} 
 \leq \frac{h(Q)}{2} 
$$
and when $n\geq 2$ one can show that $\Mass(Q) \ra \infty$ as either $n\ra \infty$ or $\det(Q) \ra \infty$, which 
shows that there are at most finitely many quadratic forms with bounded class number $h(Q) \leq B$.

In this paper we give an algorithm for enumerating all primitive positive definite $\Z$-valued maximal quadratic lattices of bounded class number, based on the exact mass formula in \cite{GHY} for maximal quadratic lattices.  
To do this we determine all quadratic spaces that could potentially support a quadratic lattice of bounded class number $h(Q) \leq B$ based on the fact that we can bound the set of primes on which such a quadratic space could have ``non-generic'' local invariants.  We then construct all quadratic spaces that could support a maximal quadratic lattice $L$ with $\Mass(L) \leq \frac{B}{2}$.   Finally, we construct a maximal lattice $L_{max}$ in this quadratic space, and taking different bases for this lattice gives a class of maximal quadratic forms $\Cls(Q)$.
This algorithm (and supporting algorithms for maximal lattices described in \cite{Hanke:fk}) has been implemented in about 10,000 lines of Python/Sage code \cite{Hanke_maximal_mass_code:kx, Hanke:uq}.   We use this code to give a complete enumeration of classes of primitive positive definite $\Z$-valued maximal quadratic forms having class number one in $n\geq 3$ variables.
%

These results are very computational in nature, and it bears mentioning that in addition to the main mathematical analysis of mass factors and their growth, much of the present work depends on having reasonably efficient implementations of basic operations of quadratic spaces and quadratic lattices over $\Q$, its completions $\Q_v$, their rings of integers $\Z_v$, and finite fields $\F_p$ (including $\F_2$ of characteristic 2).  The details of these algorithms are discussed in \cite{Hanke:uq} and \cite{Hanke:fk}.
Developing and debugging these tools takes a great deal of time, and their algorithmic development and  implementation in an open-source software environment (SAGE) is one of the fundamental contributions of this work.  This work also builds on previous quadratic forms code \cite{Ha-Sage-QF-class} developed by the author for the proof of the 290-Theorem, which are presently included in the Sage {\tt QuadraticForm()} Python class. These fundamental tools can be used by other researchers to perform computations and prove theorems about quadratic forms over $\Z$.


A notable omission is the case of binary quadratic forms of class number one (i.e. $n=2$).  This case is closely related to the question of enumerating all class number one imaginary quadratic number fields, but is actually much harder and goes under the name ``Euler's Idoneal Number problem".  (See \cite{Kani:ys} for a detailed discussion.)  In the language of quadratic fields, we must enumerate all imaginary quadratic number fields $\Q(\sqrt{-t})$ whose class group has exponent $\leq 2$.  
It is known that there are finitely many such fields, but the asymptotic lower bound for the mass is very weak here, and making it effective depends on knowing that there are no ``Siegel zeros" for any Dirichlet $L$-function $L(s, \chi_t)$.  (See e.g. \cite[p520]{Iwaniec:2004la}.)


\section{Notation and Basic Lemma}


We use the symbols $\N, \Z, \Q, \R, \C$ to denote the natural numbers (i.e. positive integers), integers, rational numbers, real numbers, and complex numbers, and denote an equivalence class of completions of $\Q$ by a {\bf place} $v$.  When $v$ is non-archimedean we often abuse notations and write $v=p$ where $p \in \N$ is the corresponding prime generating the maximal ideal of $\Z_p$.
Given a place $v$, we let $\Q_v$ denote the completion of $\Q$ at $v$, and let $\Z_v$ denote the ring of integers of $\Q_v$.  When $\Q_v = \R$ we adopt the convention that $\Z_v = \Q_v$.
We will call a number $t \in \Z \neq 0$ {\bf squarefree} if $\ord_p(t) \leq 1$ for every prime $p\mid t$.

We say that a {\bf (mod $M$) Dirichlet character} $\chi$ is a group homomorphism $\chi: (\Z/MZ)^\times \ra \C^\times$, extended to a function on $\Z/M\Z$ by taking non-units to zero.  We say that the {conductor} of Dirichlet character is the smallest modulus $\text{cond}(\chi) \in \N$ so that $\chi$ factors through a mod $f$ Dirichlet character.  If we have $M = \text{cond}(\chi)$ for some mod $M$ Dirichlet character, we say that $\chi$ is {\bf primitive}.  
For any $t \in \Z \neq 0$, we have an associated primitive Dirichlet character $\chi_t$, defined as the primitive character associated to the Dirichlet character $m \mapsto  \leg{-t}{m}$, defined by the Jacobi symbol $\leg{a}{b}$ for $b \in\N$.  
We let $\Gamma(s)$ denote the usual (analytically continued) Gamma function defined by $\Gamma(s) := \int_{\R>0} e^{-t} \cdot t^{s-1}\, dt$ when $s = x + iy \in \C$ and $x>0$.  Given a Dirichlet character $\chi$, we let $L(s, \chi) := \sum_{n \geq 1} \chi(n) n^{-s} $ denote the usual Dirichlet $L$-function.



Given a $\Z$-valued quadratic form $Q(\x) = \sum_{1\leq i \leq j \leq n} c_{ij}x_i x_j$ in $n$ variables, we define its $n \times n$ {\bf Hessian} matrix $H_Q := \frac{\partial}{\partial x_i} \frac{\partial}{\partial x_j} Q(\x)$, 
its associated {\bf Hessian bilinear form} by $B_H(\x, \y) := {}^t\x H_Q \y$, 
%
%
and its {\bf Hessian determinant} $\det_H(Q) := \det(H_Q)$ 
(which depends only on the class $\Cls(Q)$ of $Q$).
We say that a $\Z$-valued quadratic form $Q(\x)$ is {\bf primitive} if the only way of writing $Q(\x) = c \cdot Q'(\x)$ for some $c \in \Z$ and some $\Z$-valued quadratic form $Q'(\x)$ is when $c\Z = \Z$.

We define {\bf quadratic space} (over a field $K$) to be a pair $(V,Q)$ where $V$ is a finite dimensional $K$-vector space and $Q$ is a $K$-valued quadratic form on $V$, and define a {\bf quadratic lattice} (over $R$) as a finitely generated projective $R$-module $L \subset V$ with $\text{rank}_{R}(L) = \dim_K(V)$ for some quadratic space $(V,Q)$ over $K$ where $K$ is the fraction field of $R$ .  We refer to the set $Q(L)$ as the {\bf values} of $L$.  We will usually take $K = \Q$ or one of its completions $\Q_p$ or $\R$, and take the corresponding ring $R$ as the ring of integers $\Z, \Z_p,$ or $\R$.
We say the a $\Z$-valued quadratic lattice is a {\bf maximal lattice} if it is not properly contained in any other $\Z$-valued quadratic lattice in the same quadratic space.  We say that a quadratic form is a {\bf maximal form} if it arises as the function $Q(\x) = Q(\sum_{1\leq i \leq n} x_i \vec v_i)$ for some choice of (ordered) basis $\{\vec v_i\}$ for a maximal ($\Z$-valued) quadratic lattice $L$.

Given a quadratic space $(V, Q)$ over $\Q$, we say that in is {\bf isotropic} if there is some non-zero vector $\vec v \in V$ with $Q(\vec v) = 0$, and we say that $(V,Q)$ is {\bf anisotropic} otherwise.  Given any place $v$ of $\Q$ and $a,b \in \Q_v^\times$, we let $(a,b)_v \in \{\pm1\}$ denote the usual {\bf Hilbert symbol}, defined to be $1$ exactly when the quadratic form $ax^2 + by^2 - z^2$ is isotropic over $\Q_v$, and $-1$ otherwise.
For a quadratic space over $\Q_v$, we define its {\bf standard invariants}  as the triple $(n, d, c)$ defined by choosing an orthogonal basis of $V$, so that 
$Q(\x) = a_1 x_1^2 + \cdots + a_n x_n^2$, and then setting $n := dim_{\Q_v}(V)$, $d := (\prod_{1\leq i \leq n} a_i)\cdot (\Q_v^\times)^2$, and $c := \prod_{i<j} (a_i, a_j)_v \in \{\pm 1\}$.  We refer to $(n,d,c)$ respectively as the {\bf dimension}, (Gram) {\bf determinant squareclass}, and {\bf Hasse invariant} of the quadratic space.  If $(V, Q)$ is a quadratic space over $\R$, then we additionally define its {\bf signature} $\sigma := \#\{a_i > 0\} - \#\{a_i < 0\} \in \Z$.  We say that a quadratic space over $\R$ is (either positive or negative) {\bf definite} if its non-zero values are either positive or negative, but not both, and the value zero is attained only by the zero vector.











\smallskip

{\bf Acknowledgements:}  The author would like to warmly thank Robert Varley for his continuing interest in this work, and for his detailed comments.
This work was completed over many years at several different institutions, including Duke University, the Max Plank Institute (MPI) for Mathematics in Bonn, the University of Georgia (UGA), and the Mathematical Sciences Research Institute (MSRI).   The infrastructure \cite{Ha-Sage-QF-class} for computing with quadratic forms in SAGE was developed at Duke, MPI and UGA between Fall 2006 and March 2009.  The algorithms and infrastructure for computing with squareclasses, quadratic spaces, and quadratic lattices \cite{Hanke:fk, Hanke:uq, Hanke_maximal_mass_code:kx} was developed at UGA and MSRI between March 2009 and June 2011.  The final theoretical bounds and enumerative results were obtained between December 2010 and June 2011, and the author graciously thanks MSRI for their hospitality during their Spring 2011 semester in Arithmetic Statistics. 
This work was also partially supported by NSF Grant DMS-0603976.  






\section{Mass formulas for maximal lattices}\label{Sec:Mass_formula_for_maximal}
For enumerating the maximal lattices $L$ with $h(L) = 1$, we make use of the explicit mass formula  in \cite[Props 7.4 and 7.5, p121]{GHY} which gives a formula for the {\bf proper mass} 
\begin{equation} \label{Eq:GHY_mass_formula}
\Mass^+(Q):= \sum_{\Cls^+(Q') \in \Gen(Q)}\frac{1}{|\Aut^+(Q')|} 
\end{equation}
of a maximal lattice $L$ on a quadratic space over a totally real number field $F$.  
Taking $F=\Q$, this formula states that for $n\geq 3$ we have 
\begin{equation} \label{Eq:GHY_mass_formula}
\Mass^+(L) =
\begin{cases} 
\displaystyle
2^{\frac{3 - n}{2}}\prod_{k=1}^{\frac{n-1}{2}} |\zeta(1-2k)| 
    \prod_{p} \lambda_{\text{odd}, p} & \quad \text{if $n$ is odd,}\\
\displaystyle
2^{\frac{2 - n}{2}} |L(1-\tfrac{n}{2}, \chi)| \prod_{k=1}^{\frac{n-2}{2}} |\zeta(1-2k)| 
    \prod_{p} \lambda_{\text{even}, p} & \quad \text{if $n$ is even,}\\
\end{cases}
\end{equation}
where
the adjustment factors $\lambda_{\text{odd}, p}$ and $\lambda_{\text{even}, p}$ are given in the following tables, and (for any given $L$) are equal to one for all but finitely many primes $p$.






\begin{equation*} \label{Table:GHY_factors}
\begin{array}{ccc}

\bigskip
\begin{tabular}{|c||c|c|}
\multicolumn{3}{c}{\bf Table of $\lambda_{\text{odd},p}$} \\
\hline
{\multirow{2}{*}{$w$}} & {$\ord_p(\delta)$} & $\ord_p(\delta)$  \\
& even & odd \\
\hline
\hline
$1$ & $1 \vphantom{\frac{a^a}{b^b}}$ & \multirow{2}{*}{$\frac{p^\frac{n-1}{2}+ w}{2}$}   \\
\cline{1-2}
$-1$ 
& $\frac{p^{n-1}-1}{2(p+1)}$ 
\vphantom{$\frac{p^\frac{n}{2}}{p^\frac{n}{2}}$}
&   \\
\hline
\end{tabular}
\bigskip

&
\quad
&

\bigskip
\begin{tabular}{|c||c|c|c|}
\multicolumn{4}{c}{\bf Table of $\lambda_{\text{even},p}$} \\
\hline
{\multirow{2}{*}{$w$}} 
& {\multirow{2}{*}{$\ord_p(\delta)$ even}} 
& \multicolumn{2}{|c|}{$\ord_p(\delta)$ odd} \\
\cline{3-4}
& & $E_\delta$ unramified &  $E_\delta$ ramified \\
\hline
\hline
$1$ & $1$ & $1$ & {\multirow{2}{*}{$\dfrac{1}{2}$ \vphantom{\Huge $\int$} }} \\ 
\cline{1-3}
$-1$ 
& $\frac{(p^\frac{n-2}{2} - 1) (p^\frac{n}{2} - 1)}{2(p+1)}$ 
&  $\frac{(p^\frac{n-2}{2} + 1) (p^\frac{n}{2} + 1)}{2(p+1)}$ 
& 
\vphantom{$\frac{p^\frac{n}{2}}{p^\frac{n}{2}}$}
\\
\hline
\end{tabular}

\end{array}
\end{equation*}
For convenience, we define the following {\bf GHY Mass-type} labels for describing the possible local GHY Mass adjustment factors $\lambda_{\text{odd/even},p}$ above:

$$
\begin{array}{ccc}

\bigskip
\begin{tabular}{|c||c|c|}
\multicolumn{3}{c}{\bf Odd dim'l GHY Mass-types} \\
\hline
{\multirow{2}{*}{$w$}} & {$\ord_p(\delta)$} & $\ord_p(\delta)$  \\
& even & odd \\
\hline
\hline
$1$ & --  & II$+$   \\
\hline
$-1$ 
& I 
&   II$-$ \\
\hline
\end{tabular}
\bigskip

&
\quad
&

\bigskip
\begin{tabular}{|c||c|c|c|}
\multicolumn{4}{c}{\bf Even dim'l GHY Mass-types} \\
\hline
{\multirow{2}{*}{$w$}} 
& {\multirow{2}{*}{$\ord_p(\delta)$ even}} 
& \multicolumn{2}{|c|}{$\ord_p(\delta)$ odd} \\
\cline{3-4}
& & $E_\delta$ unramified &  $E_\delta$ ramified \\
\hline
\hline
$1$ & -- & -- & {\multirow{2}{*}{III}} \\
\cline{1-3}
$-1$ 
& I 
&  II  
& 
\\
\hline
\end{tabular}

\end{array}
$$

\begin{rem}
Since there is exactly one genus of maximal ($\Z$-valued) lattices in a given quadratic space $(V,Q)$ (see \cite[\S30.1, pp172-3]{Shimura:2010uq}), the proper mass can be thought of as being determined by the underlying quadratic space.  Therefore $\Mass^+(L)$  can be expressed in terms of the standard local invariants $(n, d_v, c_v)$ of the quadratic space. 
\end{rem}

One of our main goals is to explicitly enumerate the results of the following theorem:

\begin{thm}
There are only finitely many (classes of) maximal $\Z$-valued quadratic lattices in $n\geq 3$ variables with a priori bounded class number.
\end{thm}

\begin{proof}
When $n$ is odd, there are only finitely many primes $p$ where $0 < \lambda_{\text{odd}, p} < B'$ for any $B'$.  Therefore there are only finitely many products $\prod_p \lambda_{\text{odd}, p}$ and finitely many masses $\Mass^+(L)$ below any given bound. 
This proves the claim when $n$ is odd since $\Mass^+(L) \leq h(L)$.

When $n$ is even, the number of primes $p$ with $\lambda_{\text{even}, p} = \frac12$ may be arbitrarily large (as we vary the quadratic space $(V,Q)$), but these factors will cancel with the growth of the quadratic twist $L(1-\frac{n}{2}, \chi)$.  From the definition of $E_\delta$ we see that $E_\delta/\Q_p$ is ramified iff $p \mid \Cond(\chi)$, so if $t$ is the number of prime divisors of $q := \Cond(\chi)$, we can write (\ref{Eq:GHY_mass_formula}) 
as 
\begin{equation} \label{Eq:Mass_for_even_n}
\Mass^+(L) =
2^{\frac{2 - n}{2}} \frac{|L(1-\tfrac{n}{2}, \chi)|}{2^t} \prod_{k=1}^{\frac{n-2}{2}} |\zeta(1-2k)| 
    \prod_{p \nmid \Cond(\chi)} \hspace{-.1in} 
    \lambda_{\text{even}, p}.
\end{equation}
Now Lemma \ref{Lemma:bounding_twists} states that $\frac{|L(1-\tfrac{n}{2}, \chi)|}{2^t}$ is bounded for only finitely many conductors $q$, and the product of the remaining $\lambda_{\text{even}, p}$ with $p\nmid q$ is also bounded for only finitely many set of primes.  The result now follows by applying the same argument as when $n$ is odd.
%
%
%
\end{proof}


The following Lemma gives a simple relation between the mass 
and the proper mass that will be useful for us later.
%
\begin{lem} Suppose that $Q$ is a totally definite quadratic form over a totally real number field $F$.  Then
$$
\Mass^+(Q) = 2\cdot\Mass(Q).
$$
\end{lem}
\begin{proof}
We notice that $\Cls(Q)$ is a union of at most 2 proper classes, and $\Cls(Q) = \Cls^+(Q) \iff Q$ has an automorphism of determinant $-1$ $\iff \Aut(Q) = 2 \Aut^+(Q)$.  Therefore in either case we have  
\begin{align*}
\Mass^+(Q) 
&= \sum_{\Cls^+(Q'') \in \Gen(Q)} \frac{1}{\#\Aut^+(Q'')}
= \sum_{\Cls(Q') \in \Gen(Q)} \sum_{\Cls^+(Q'') \in \Cls(Q')} \frac{1}{\#\Aut^+(Q'')} \\
& = \sum_{\Cls(Q') \in \Gen(Q)} \frac{2}{\#\Aut(Q')} 
= 2\cdot \Mass(Q).
\end{align*}
\end{proof}


\section{Translation of local invariants}

In this section we describe the explicit translation between the local quadratic space invariants $(n, \delta_v, w_v)$ in \cite{GHY} and the standard local invariants $(n, d_v, c_v)$.

\begin{lem} \label{Lem:GHY_and_Std_invariants}
Given a quadratic space over a non-archimedean local field $\Q_v$, there is an explicit bijection between the local invariants $(n, \delta_v, w_v)$ of \cite{GHY} and the standard local invariants $(n, d_v, c_v)$, explicitly given in the Table local invariants below.
\end{lem}

\begin{proof}
By induction with the formula $(n,d,c) \oplus (n', d', c') = (n + n', d\cdot d', c\cdot c' \cdot (d, d')_v)$ for the standard local invariants of the direct sum of two quadratic spaces, one can show that $\Hyp^r$ has standard invariants $(n,d,c) = (2r, (-1)^r, (-1,-1)_v^{\lfloor \frac{r}{2} \rfloor})$.  Then by applying this formula again to the five spaces below, we can generate a table that quickly determines the standard invariants from the GHY invariants.  To translate in the opposite direction we compute the standard invariants $(n_\text{aniso}, d_\text{aniso}, c_\text{aniso})$ of the maximal anisotropic subspace.



One can easily see that distinct GHY triples $(n, \delta, w)$ give rise to distinct standard triples $(n, d, c)$ in the first three spaces, and the remaining two odd dimensional cases follow by taking cases on the parity of $r$.
\end{proof}


\centerline{
\begin{tabular}{| c | c | c |} 
\multicolumn{3}{c}{\bf Tables of GHY and Standard local invariants} \\
\hline
\multirow{2}{*}{Quadratic  Space}  & \multirow{2}{*}{Dim.} & Aniso. \\
 & & dim. \\
\hline
\hline
$\Hyp^r$ & $2r$ & 0   \\ 
$\Hyp^{r-2} \oplus D$ & $2r$ & 4 \\ 
$\Hyp^{r-1} \oplus \al(x^2 - \delta y^2)$ & $2r$ & 2 \\ 
\hline
$\Hyp^r \oplus \delta x^2$ & $2r+1$ & 1  \\ 
$\Hyp^{r-1} \oplus -\delta(ax^2  + by^2 + abz^2 )$ & $2r+1$ & 3 \\
\hline
\end{tabular}
}
\bigskip

%

%
\centerline{
\begin{tabular}{|  c | c || c | c | c || c | c |}
\hline
\multirow{2}{*}{Dim.} & Aniso. &
\multicolumn{3}{c||}{GHY} & \multicolumn{2}{c|}{Standard}  \\
\cline{3-7}
 & dim. & Type & $\delta$ & $w$ & $d$ & $c$ \\
\hline
\hline
 $2r$ & 0   
    & -- & 1 & 1 & $(-1)^r$ & $(-1,-1)_v^{\lfloor \frac{r}{2} \rfloor}$ \\ 
 $2r$ & 4
    & I & 1 & $-1$ & $(-1)^{r}$ & $-(-1,-1)_v^{\lfloor \frac{r}{2} \rfloor}$ \\ 
 $2r$ & 2
    & II, III & $\delta$ & $c_\text{aniso}$ & $(-1)^r \delta$ 
    & $(-1,-1)_v^{\lfloor \frac{r-1}{2} \rfloor} \cdot w \cdot ((-1)^{r-1}, -\delta)_v$ \\ 
\hline
 $2r+1$ & 1 
    &--, I& $\delta$ & 1 & $(-1)^r \delta$ 
    & $(-1,-1)_v^{\lfloor \frac{r}{2}\rfloor} \cdot ((-1)^{r}, \delta)_v$ \\ 
 $2r+1$ & 3
    &II$\pm$& $\delta$ & $-1$ & $(-1)^{r} \delta$ 
    & $-(-1,-1)_v^{\lfloor \frac{r-1}{2}\rfloor} \cdot ((-1)^{r-1}, -\delta)_v \cdot (-1, \delta)_v$ \\ 
%
\hline
\end{tabular}
}
\bigskip

To determine which collections of local spaces can be assembled into a global rational space, we rephrase the local-global conditions in terms of the GHY local invariants.

\begin{lem}
Suppose that $(V, Q)$ is a positive definite quadratic space over $\Q$.  Then the GHY local invariants satisfy the product formula 
\begin{equation} \label{Eq:GHY_prod_formula}
\prod_p w_p = 
(-1)^{\lfloor \frac{n}{4} \rfloor}
\prod_p (-1, \delta)_p^{\lfloor \frac{n+1}{2} \rfloor}.
%
\end{equation}
Moreover, a collection of local quadratic spaces $\{(V_p, Q_p)\}_{p}$ of common dimension $n \geq 2$ can be assembled into a positive definite rational space over $\Q$ iff there exists some $\Delta \in \Q^\times>0$ so that $\Delta (\Q_p^\times)^2 = (-1)^{\lfloor \frac{n}{2} \rfloor} \delta_p (\Q_p^\times)^2$ for all primes $p$, and equation (\ref{Eq:GHY_prod_formula}) is satisfied.
\end{lem}


\begin{proof}
Since $(V, Q)$ is positive definite, we have $c_v = 1$ for all archimedean $v$.
From the table above, we express $\prod_v c_v  = \prod_p c_p = 1$ as a product 
$\prod_p w_p$ with some extra Hilbert symbol factors at each place.  By taking cases on the parity of $n$ and of $r$, and factoring out $\prod_v (-1,-1)^{\lfloor \frac{r}{2} \rfloor} = 1$, we obtain the formula above.
(In the case where $n$ and $r$ are even  the anisotropic dimension is 0 or 4$\iff \delta = 1$, so we see that the product $\prod_p(-1, \delta)$ agrees with the same product over primes with anisotropic dimension 2.)

The existence statement is a rephrasing of \cite[Thrm 1.3, p77]{Cassels:1978aa} using the GHY local invariants, and here $\Delta \geq 0$ guarantees that we get the correct local determinant squareclass at $\infty$.  Here $d = (-1)^r \delta = (-1)^{\lfloor \frac{n}{2} \rfloor} \delta$.
\end{proof}

\section{Bounding Eligible Quadratic Twists} \label{Sec:Twist_Bounds}

\subsection{Functional equation for primitive quadratic twists of the Dirichlet $L$-function}

Let $\chi$ be a primitive quadratic character of conductor $q>1$.  Then we define the completed $L$-function
$$
\Lambda(s, \chi) := q^{\frac{s}{2}} \cdot \pi^{\frac{-s}{2}} \Gamma(\tfrac{s}{2}) \cdot  L(s, \chi)
$$
which satisfies the functional equation
$$
\Lambda(s, \chi) = \frac{\tau(\chi)}{\sqrt{q}} \cdot \Lambda(1-s, \overline{\chi})
$$
where $\tau(\chi) := \sum_{0 < m < q := \text{cond}(\chi)} \chi(m) e^\frac{2\pi i m}{q}$ is the usual Gauss sum.

\begin{lem} \label{Lemma:bounding_twists}
Suppose that $\chi$ is a primitive quadratic Dirichlet character of conductor $q$ with $\chi(-1) = (-1)^a$, and $r \in \N \geq 2$ with $r \equiv a \pmod{2}$.  Let $t$ be the number of distinct positive prime factors of $q$.  Then for any constant $K \in \R>0$ we have that 
$$\frac{|L(1-r, \chi)|}{2^t} > K$$
whenever 
$q$ has a positive divisor $q'$ satisfying 
$$
\prod_{p \mid q'} \frac{p^{\ord_p(q') (r - \frac{1}{2})}}{2} 
    > \frac{K \cdot (2\pi)^{r} \zeta(r)}{2\,\Gamma(r) \, \zeta(2r)}.
$$ 
In particular, this holds if $q$ is divisible by any prime power $p^\nu>1$ where
$$
p^\nu >
\(\frac{K \cdot (2\pi)^{r} \zeta(r)}{\Gamma(r) \zeta(2r)}\)^{\frac{2}{2r-1}}.
$$
\end{lem}






\begin{proof}
The functional equation 
$$
L(1-s, \overline{\chi}) = \ve(\chi) \cdot 2 \cdot (2\pi)^{-s} 
	q^{s - \frac12} \cos(\tfrac{\pi(s-a)}{2}) \Gamma(s) L(s, \chi)
$$
for Dirichlet $L$-functions $L(s, \chi)$ of primitive characters $\chi$ with conductor $q$ 
where $|\ve(\chi)| = 1$ and $s := r \equiv a \pmod{2}$ gives for quadratic $\chi =  \overline{\chi}$
the relation  
$$
2 \cdot (2\pi)^{-r} 
	q^{s - \frac12}  \Gamma(s) |L(r, \chi)| 
= |L(1-r, \chi)|. 
$$
Since $r \in \R > 1$ we have that $L(r, \chi) \geq 
 \frac{\zeta(2r)}{\zeta(r)}$ and so 
$$
q^{r - \frac12} \cdot \frac{2\Gamma(r) \zeta(2r)}{(2\pi)^{r} \zeta(r)} 
\leq 
|L(1-r, \chi)|. 
$$
Thus 
$$
K <
\frac{q^{r - \frac12}}{2^t} \cdot \frac{2\Gamma(r) \zeta(2r)}{(2\pi)^{r} \zeta(r)} 
\leq 
\frac{|L(1-r, \chi)|}{2^t}
$$
when 
$$
\prod_{p\mid q} \frac{p^{\ord_p(q) (r - \frac{1}{2})}}{2} > \frac{K \cdot (2\pi)^{r} \zeta(r)}{2\Gamma(r) \zeta(2r)}.
$$
Since $q > 1$ we see that the Euler product has at least one factor, and $r\geq 2$ tells us that each Euler factor is  $>1$.  Since the left side of this inequality increases as $q$ becomes more divisible, this holds if we replace $q$ by any $q' \mid q$, proving the theorem.
%
\end{proof}




%




\section{Strategy and explicit bounds for enumeration} \label{Sec:Bounds_for_maximal_lattices}
In this section we describe an explicit algorithm for enumerating all maximal $\Z$-valued positive definite quadratic lattices $L$ of fixed rank $n\geq 3$ with class number $h(L) \leq B$, for any given $B \in \R>0$.  We do this by computing explicit (refinements of the) bounds $B'$ and $B''$ in the implications
\begin{align} \label{Imp:mass_eligible}
h(L) \leq B & \quad \xRightarrow{\hphantom{aaaaa}} \quad
\Mass^+(L) = \Mass^+(V,Q) \leq B \notag \\
& \quad \xRightarrow{\text{Lemma \ref{Lemma:bounding_twists}}\hphantom{a}} \quad 
\text{$\Cond(\chi) \leq B'$ 
when $n$ is even} \\ 
& \quad\xRightarrow{\text{ fixed $\chi$}\hphantom{a}}\quad
\prod_p \lambda_{\text{odd/even}, p} \leq B''.  \notag
\end{align}

\begin{defn}
We say that an object (e.g. quadratic space, character, local GHY mass-type, local invariant tuple or quadratic lattice) is {\bf mass-eligible} if it arises in the implication (\ref{Imp:mass_eligible}) above.
\end{defn}



Our algorithm for computing those maximal quadratic lattices $L$ as above with $h(L) \leq B$ proceeds in the following steps, which enumerates the mass-eligible objects in implication (\ref{Imp:mass_eligible}) in reverse order:
\begin{enumerate}
\item[1)] Evaluate the zeta product exactly.
\item[2)] When $n$ is even, determine the finitely many mass-eligible quadratic characters $\chi$.
\item[3)] Determine the finitely many mass-eligible local GHY mass-types.
\item[4)] Construct all possible rational quadratic spaces $(V, Q)$ for each mass-eligible tuple of local GHY mass-types.
\item[5)] Construct a maximal $\Z$-valued lattice $L$ on each mass-eligible $(V, Q)$.
\item[6)] Compute the class number $h(L)$, and check if $h(L) \leq B$.
\end{enumerate}

\medskip 
We now give the explicit bounds used to enumerate objects in every step above.  The special values of $\zeta(1-2k)$ when $k \in \N$ are given by the usual Bernoulli number formula 
$$
\zeta(1-2k) = \frac{-B_{2k}}{2k} 
\qquad \text{where $\quad\frac{x}{e^{x} - 1} =: \sum_{k=0}^\infty B_k \frac{x^k}{k!}$}.
$$

For steps 2) and 3), the following lemma is useful.
\begin{lem} \label{Lem:lambda_below_one}
Suppose that $Q$ is a definite maximal $\Z$-valued quadratic form.  Then the GHY adjustment factors in the Tables of $\lambda_{\text{odd/even},p}$ 
satisfy
$$
\lambda_{\text{odd/even},p} < 1  \implies
\begin{cases}
\text{$n=3$ and $p=2$, with GHY Mass-type I or II$-$}, \\
\text{$n=4$ and $p=2$, with GHY Mass-type I}, \\
\text{$n$ is even with GHY Mass-type III}.
\end{cases}
$$
In each of these cases, we have $\lambda_{\text{odd/even},p} = \tfrac12$.
\end{lem}

\begin{proof}
If $n$ is odd, then we see that the minimal $\lambda_{\text{odd},p}$ will be attained with GHY mass-type II$-$, and this is $<1 \iff n=3$ and $ p=2$.  In this case the GHY mass-type I has the same value, $\lambda_{\text{odd},2} = \frac{1}{2}$.

When $n$ is even and GHY mass-type $\neq$ III, we see that the minimal $\lambda_{\text{even},p}$ will be attained with GHY mass-type I, which is $<1 \iff n=4$ and $ p=2$. Increasing either $p$ or $n$ gives $\lambda_{\text{even},p} \geq 1$, as does GHY mass-type II in this case.
\end{proof}

To perform step 2), we use the following explicit upper bound.
%
\begin{lem} \label{Lem:Twist_upper_bound}
Suppose that $L$ is a maximal $\Z$-valued definite quadratic lattice of even rank $n\geq 4$ with $h(L)\leq B$.  Then the primitive quadratic character $\chi(\cdot) := \leg{(-1)^{n/2}\det(Q)}{\cdot}$ satisfies the bound
\begin{equation} \label{Eq:Eligible_twist_bound}
\frac{|L(1-\frac{n}{2}, \chi)|}{2^t}  \leq \frac{B \cdot 2^\frac{n-2}{2}}{\prod_{k=1}^\frac{n-2}{2} |\zeta(1-2k)|} 
\cdot 
\begin{cases}
2 & \text{if $n=4$,} \\
1 & \text{if $n \geq 6$,} \\
\end{cases}
\end{equation}
where $t$ in the number of prime divisors of $\Cond(\chi)$.
There are only finitely many $\chi$ for which (\ref{Eq:Eligible_twist_bound}) holds, and these mass-eligible $\chi$ can be explicitly enumerated by the conductor bounds in Lemma \ref{Lemma:bounding_twists}.
%
\end{lem}

\begin{proof} By rewriting (\ref{Eq:GHY_mass_formula}) as (\ref{Eq:Mass_for_even_n}), we see that 
$$
B \geq h(L) \geq \Mass^+(L) =
2^{\frac{2 - n}{2}} \frac{|L(1-\tfrac{n}{2}, \chi)|}{2^t} \prod_{k=1}^{\frac{n-2}{2}} |\zeta(1-2k)| 
    \prod_{p \nmid \Cond(\chi)} \hspace{-.1in} 
    \lambda_{\text{even}, p}.
$$ 
However by Lemma \ref{Lem:lambda_below_one} we know that 
$$
\prod_{p \nmid \Cond(\chi)} \hspace{-.1in} \lambda_{\text{even}, p}
\quad =
\prod_{\substack{p \text{ with} \\ \text{Mass-type $\neq$ III}}} \hspace{-.2in} \lambda_{\text{even}, p} 
\quad \geq \quad 
\begin{cases}
\frac{1}{2} & \text{when $n=4$,} \\
1 & \text{when $n\geq 6$,} \\
\end{cases}
$$
which proves the desired bound.  The finiteness statement follows from the prime-power divisibility bound in Lemma \ref{Lemma:bounding_twists}, and the fact that each conductor supports a unique quadratic character of that conductor.
\end{proof}

In step 3) we must construct all possible tuples $T$ (indexed by primes $p\in\N$) of local GHY mass-types $\lambda_{\text{odd/even}, p}$ whose proper masses (as computed using (\ref{Eq:GHY_mass_formula})) satisfy $\Mass^+(T) \leq B$.  This is done by the following Lemma.
\begin{lem}
Suppose that $L$ is a maximal $\Z$-valued definite quadratic lattice of rank $n\geq 3$ with $h(L)\leq B$, with the quadratic character $\chi$ of Lemma \ref{Lem:Twist_upper_bound} specified if $n$ is even.  Let 
$$
B'' := 
\begin{cases}
\displaystyle  
\frac{2^{\frac{n - 3}{2}}}{\prod_{k=1}^{\frac{n-1}{2}} |\zeta(1-2k)|} 
     & \quad \text{if $n$ is odd,}\\
\displaystyle
\frac{2^{\frac{n - 2}{2}}}{|L(1-\tfrac{n}{2}, \chi)| \prod_{k=1}^{\frac{n-2}{2}} |\zeta(1-2k)|}
     & \quad \text{if $n$ is even,}\\
\end{cases}
$$
and 
$$
\epsilon = 
\begin{cases}
2 & \quad \text{if $3 \leq n \leq 4$,}\\
1 & \quad \text{if $n \geq 5$.}\\
\end{cases}
$$
Then the associated adjustment factors $\lambda_{\text{odd/even}, p}$ satisfy
\begin{equation} \label{Eq:lambda_single_bound}
\lambda_{\text{odd/even}, p} \leq \epsilon B'',
\end{equation}
and also 
\begin{equation} \label{Eq:lambda_all_bound}
%
\prod_p \lambda_{\text{odd/even}, p} \leq B''.
\end{equation}
\end{lem}

\begin{proof} This follows from (\ref{Eq:GHY_mass_formula}) combined with Lemma \ref{Lem:lambda_below_one}. 
\end{proof}

\noindent
Explicitly, this is done by first enumerating all eligible $\lambda_{\text{odd/even}, p}$ adjustment factors for all eligible primes $p$ using (\ref{Eq:lambda_single_bound}), and then checking all possible products for being mass-eligible using (\ref{Eq:lambda_all_bound}), to obtain a finite list of mass-eligible GHY local invariants.  Note that here several different tuples of local invariants can be associated to the same tuple of mass-eligible GHY mass-types.

In step 4) we use Lemma \ref{Lem:GHY_and_Std_invariants} convert each tuple of GHY local invariants to standard invariants, and then use a constructive Hasse-Minkowski procedure to give a diagonal representative for the unique (positive definite) quadratic space having these local invariants, if one exists.  This procedure is based on the constructive existence proof in \cite[Ch. 6, \S7, pp85-86]{Cassels:1978aa}.

The algorithm for step 5) is described in detail in \cite{Hanke:fk}, and consists of three steps: finding a $\Z$-valued quadratic lattice, finding a maximal $\Z$-valued Hessian-bilinear lattice, and finally finding a maximal $\Z$-valued quadratic lattice.  

Step 6) computes the class number of $L$ by Kneser's method of neighboring lattices at several primes $p \nmid 2\det(L)$.  This algorithm is described in \cite[\S1.10, pp22-23]{Hanke:AWS2009}. 

\bigskip 
By carrying out these computations, we have that 

\begin{thm}
There are exactly 115 maximal $\Z$-valued positive definite quadratic forms $Q$ in $n \geq 3$ variables with class number $h(Q) = 1$.  These are enumerated explicitly in \S\ref{Sec:CN1_Tables}.
\end{thm}



\begin{rem}[Related Results and Generalizations]
An exact mass formula for maximal quadratic lattices was given by Shimura in \cite{Shimura:1999ad} by computing residues of certain Eisenstein series and later by the author  \cite{Hanke:1999aa, Hanke:2005mr} using Shimura's local computations via the Tamagawa number formalism.  Shimura's formula was later generalized and simplified in \cite{GHY} by using motivic results of Gross and Bruhat-Tits theory.  Since the results of \cite{GHY} are stated for (totally) definite maximal quadratic and hermitian lattices over number fields, the present enumerative results could also be generalized to that context as well.
\end{rem}

\begin{rem}[Remarks on Ternary forms]
In \cite{Shimura:2006ac} Shimura has computed that there are 64 classes of primitive maximal ternary quadratic forms based on the correspondence with quaternion orders (e.g. see \cite[\textsection 14]{Eichler:1974ta} or \cite{Voigt:ys}), where the quaternion orders of small type number have already been enumerated by Pizer \cite{Pizer:1973vn} using Selberg's trace formula.  
\end{rem}

In the ternary case, one can use a simple divisibility argument to bound the primes $p$ dividing the (Hessian) determinants of forms with class number one. 

\begin{thm} \label{Thrm:class_number_one_by_divisiblilty}
If a prime $p$ divides the Hessian determinant $\det_H(Q)$ of a primitive maximal $\Z$-valued positive definite quadratic form in 3 variables with class number $h(Q) = 1$, then $p \leq 23$.
\end{thm}

\begin{proof}
When $n=3$ the exact mass formula of Section \ref{Sec:Mass_formula_for_maximal} takes the form
$$
\Mass^+(Q) 
= |\zeta(-1)| \prod_p \lambda_{\text{odd}, p} 
= \frac{1}{12} \prod_{p\mid \det_H(Q)} \lambda_{p}
$$
where the non-trivial adjustment factors are given by $\lambda_p := \lambda_{\text{odd}, p} = \frac{p \pm 1}{2}$.  Notice that if $p\neq2$ then $\lambda_p \in \N$, and that $\lambda_{2} = 1$ or $\frac32$.

If $h(Q) = 1$ then $\Mass^+(Q)  = \frac{1}{\Aut^+(Q)} \in \frac{1}{\N}$.  This implies that any odd prime $p\mid \det_H(Q)$ must satisfy 
$$
\frac{p \pm 1}{2} \cdot \lambda_2 \mid 12
$$
which implies that $\frac{p \pm 1}{2}$ divides either 8 or 12.  In either case we must have $\frac{p \pm 1}{2} \leq 12$, showing that $p \leq 23$.
\end{proof}

\begin{rem}[Optimality of the Ternary Divisibility bound]
In the tables of class number one maximal positive definite lattices in Section \ref{Sec:CN1_Tables}, we see that $p=23$ does arise as a factor of the Hessian determinant of some ternary quadratic form, so the bound given in Theorem \ref{Thrm:class_number_one_by_divisiblilty} is an optimal bound.
\end{rem}


\newpage

\section{Tables of primitive maximal $\Z$-valued quadratic forms in $n\geq 3$ variables}\label{Sec:CN1_Tables}

\begin{sagesilent}
load("/Users/jonhanke/Dropbox/SAGE/qf_init.sage")
\end{sagesilent}

\begin{sagesilent}

def my_split(s, seps):
    res = [s]
    for sep in seps:
        s, res = res, []
        for seq in s:
            res += seq.split(sep)
    return res


def maximal_QF_table(n, index_start=None, index_stop=None, terms_per_line=None):
    # Make the range of the table
    if index_start == None:
        index_start = 0
    if index_stop == None:
        index_stop = len(CN1_maximal_lattice_list[n]) - 1
    index_list = range(index_start, index_stop +1)


    # start of the table
    s  = [ r"""\begin{tabular}{|c|c|c|c|}""" ]

    # add the heading
    s.append(r" \hline ")
    s.append(r" \text{\#} ")
    s.append(r" & ")
    s.append(r" \text{Maximal CN1 forms in $" + str(n) + r"$ variables} & ")
    s.append(r" \text{det} ")
    s.append(r" & ")
    s.append(r" \text{det factors} ")
    s.append(r" \\ ")
    s.append(r" \hline ")

    ## Setup the polynomial ring
    R1 = PolynomialRing(ZZ, 10, 'abcdefghij')

    ## Add all of the entries    
    for i in index_list:

        ## Setup the lattice
        L = CN1_maximal_lattice_list[n][i]
        Q = L.quadratic_form__integral().lll()
        dim = Q.dim()

        ## Compute the string for the quadratic form
        Q_poly = Q.base_change_to(R1)(R1.gens()[:dim])
        Q_str = str(Q_poly).replace(r"*", r'')

        ## Determine the number of terms per line, and the number of lines
        if terms_per_line == None:
            Q_str_line_list = [Q_str]	
        else:
            num_of_terms = len(my_split(Q_str, ['+', '-']))
            Q_str_line_list = []

            ## Pull out all lines with "terms_per_line" number of terms
            while num_of_terms >= terms_per_line:

	       ## Find the string index in Q_str for the first "terms_per_line" number of terms
                Q_str_index = 0
                tmp_num_of_terms = 0
                while (Q_str_index < len(Q_str)) and (tmp_num_of_terms < terms_per_line):
                    if (Q_str[Q_str_index] in ['+', '-']):
                        tmp_num_of_terms += 1
                    Q_str_index += 1

                ## Deal with finishing the line with the last term
                if (Q_str_index == len(Q_str)):
                    tmp_num_of_terms += 1
                
                ## Pull out this many terms of Q_str, and continue
                Q_str_line_list.append(Q_str[0:Q_str_index])
                Q_str = Q_str[Q_str_index:]
                num_of_terms -=  terms_per_line

            ## Output the remaining few terms as the last line, if there are any left.
            if (num_of_terms > 0):
                Q_str_line_list.append(Q_str)


        ## Add the first line for each entry
        num_of_lines = len(Q_str_line_list)
        s.append(r"\multirow{" + str(num_of_lines) + r"}{*}{")
        s.append(str(i + 1))
        s.append(r"} & ")
        s.append(r" $" + Q_str_line_list[0] + r"$ ")
        s.append(r" & ")
        s.append(r"\multirow{" + str(num_of_lines) + r"}{*}{")
        s.append(r" $" + str(Q.det()) + r"$ ")
        s.append(r"} & ")
        Q_det_factor_str = str(factor(Q.det())).replace(r"*", r' \cdot ')
        s.append(r"\multirow{" + str(num_of_lines) + r"}{*}{")
        s.append(r" $" + str(Q_det_factor_str) + r"$ ")
        s.append(r"} \\ ")
        
        ## Add the remaining lines for each entry		    
        for j in range(1, num_of_lines):
            s.append(r" & ")
            s.append(r" $" + Q_str_line_list[j] + r"$ ")
            s.append(r" & ")
            s.append(r" & ")
            s.append(r" \\ ")
            
        ## Add the separating hline
        if terms_per_line != None:
            s.append(r" \hline ")


    ## Add the last line and return
    if terms_per_line == None:
        s.append(r" \hline ")
    s.append(r"\end{tabular} ")

    return ''.join(s)

\end{sagesilent}

\begin{center}
$
\sage{maximal_QF_table(3, 0, 44)}
$
\end{center}

\newpage

\begin{center}
$
\sage{maximal_QF_table(3, 45)}
$
\end{center}

\begin{center}
$
\sage{maximal_QF_table(4)}
$
\end{center}


\begin{center}
$
\sage{maximal_QF_table(5, terms_per_line=10)}
$
\end{center}

\medskip

\begin{center}
$
\sage{maximal_QF_table(6, terms_per_line=9)}
$
\end{center}

\medskip

\begin{center}
$
\sage{maximal_QF_table(7, terms_per_line=8)}
$
\end{center}

\medskip

\begin{center}
$
\sage{maximal_QF_table(8, terms_per_line=8)}
$
\end{center}

\medskip

\begin{center}
$
\sage{maximal_QF_table(9, terms_per_line=8)}
$
\end{center}

\medskip

\begin{center}
$
\sage{maximal_QF_table(10, terms_per_line=8)}
$
\end{center}

\bigskip
\hrule
\bigskip

\begin{center}
  \begin{tabular}{ | c || c | c | c || c | c | c | c | c | c | c | c | }
  \hline
    \multicolumn{12}{|c|}{\# of classes of primitive maximal $\Z$-valued} \\ 
    \multicolumn{12}{|c|}{positive definite quadratic forms $Q$ with $h(Q) = 1$} \\ \hline \hline 
    rank $n$ & 0 & 1 & 2 & 3 & 4 & 5 & 6 & 7 & 8 & 9 & 10 \\ \hline
    \# & 1 & 1 & ?? & 64 & 20 & 12 & 10 & 5 & 2 & 1 & 1 \\ \hline
  \end{tabular}
\end{center}

\newpage

%


%


\bibliographystyle{alpha}	
\bibliography{maximal_CN1_refs}		

\newcommand{\etalchar}[1]{$^{#1}$}
\begin{thebibliography}{GHY01}

\bibitem[Bha00]{Bhargava:2000kx}
Manjul Bhargava.
\newblock On the {C}onway-{S}chneeberger fifteen theorem.
\newblock In {\em Quadratic forms and their applications ({D}ublin, 1999)},
  volume 272 of {\em Contemp. Math.}, pages 27--37. Amer. Math. Soc.,
  Providence, RI, 2000.

\bibitem[Cas78]{Cassels:1978aa}
J.~W.~S. Cassels.
\newblock {\em Rational quadratic forms}, volume~13 of {\em London Mathematical
  Society Monographs}.
\newblock Academic Press Inc. [Harcourt Brace Jovanovich Publishers], London,
  1978.

\bibitem[Coh80]{Cohn1980}
Harvey Cohn.
\newblock {\em Advanced number theory}.
\newblock Dover Publications Inc., New York, 1980.
\newblock Reprint of {\it A second course in number theory}, 1962, Dover Books
  on Advanced Mathematics.

\bibitem[Eic74]{Eichler:1974ta}
Martin Eichler.
\newblock {\em Quadratische {F}ormen und orthogonale {G}ruppen}.
\newblock Springer-Verlag, Berlin, 1974.
\newblock Zweite Auflage, Die Grundlehren der mathematischen Wissenschaften,
  Band 63.

\bibitem[Gau86]{Gauss:1986uq}
Carl~Friedrich Gauss.
\newblock {\em Disquisitiones arithmeticae}.
\newblock Springer-Verlag, New York, 1986.
\newblock Translated and with a preface by Arthur A. Clarke, Revised by William
  C. Waterhouse, Cornelius Greither and A. W. Grootendorst and with a preface
  by Waterhouse.

\bibitem[Ger72]{Gerstein:1972ya}
Larry~J. Gerstein.
\newblock The growth of class numbers of quadratic forms.
\newblock {\em Amer. J. Math.}, 94:221--236, 1972.

\bibitem[Ger73]{Gerstein:1973wb}
Larry~J. Gerstein.
\newblock Orthogonal splitting and class numbers of quadratic forms.
\newblock {\em J. Number Theory}, 5:332--338, 1973.
\newblock The arithmetical theory of quadratic forms, I (Proc. Conf., Louisiana
  State Univ., Baton Rouge, La., 1972; dedicated to Louis Joel Mordell).

\bibitem[GHY01]{GHY}
Wee~Teck Gan, Jonathan~P. Hanke, and Jiu-Kang Yu.
\newblock On an exact mass formula of {S}himura.
\newblock {\em Duke Math. J.}, 107(1):103--133, 2001.

\bibitem[Gol85]{Goldfeld:1985zr}
Dorian Goldfeld.
\newblock Gauss's class number problem for imaginary quadratic fields.
\newblock {\em Bull. Amer. Math. Soc. (N.S.)}, 13(1):23--37, 1985.

\bibitem[Hana]{Ha-Sage-QF-class}
Jonathan Hanke.
\newblock {Q}uadratic {F}orms library for the {S}age computer algebra system.
\newblock Tickets \#4470, 5418, and 5954 at \texttt{http://trac.sagemath.org/}.
\newblock ($\approx$ 22,000 lines of Python code).

\bibitem[Hanb]{Hanke:fk}
Jonathan~P. Hanke.
\newblock Algorithms for computing a maximal lattice in bilinear and quadratic
  spaces.
\newblock (In Progress).

\bibitem[Hanc]{Hanke_maximal_mass_code:kx}
Jonathan~P. Hanke.
\newblock Maximal mass formula routines for quadratic forms in the {SAGE}
  computer algebra system.

\bibitem[Hand]{Hanke:uq}
Jonathan~P. Hanke.
\newblock Quadratic forms library {II} for the {SAGE} computer algebra system.

\bibitem[Han99]{Hanke:1999aa}
Jonathan Hanke.
\newblock {\em An exact mass formula for quadratic forms over number fields}.
\newblock PhD thesis, Princeton University, 1999.

\bibitem[Han05]{Hanke:2005mr}
Jonathan Hanke.
\newblock An exact mass formula for quadratic forms over number fields.
\newblock {\em J. Reine Angew. Math.}, 584:1--27, 2005.

\bibitem[Han11]{Hanke:AWS2009}
Jonathan Hanke.
\newblock {N}otes on ``{Q}uadratic {F}orms and {A}utomorphic {F}orms'' from the
  2009 {A}rizona {W}inter {S}chool.
\newblock \texttt{http://arxiv.org/abs/1105.5759}, 2011.
\newblock (Submitted).

\bibitem[IK04]{Iwaniec:2004la}
Henryk Iwaniec and Emmanuel Kowalski.
\newblock {\em Analytic number theory}, volume~53 of {\em American Mathematical
  Society Colloquium Publications}.
\newblock American Mathematical Society, Providence, RI, 2004.

\bibitem[JKS97]{Jagy:1997aa}
William~C. Jagy, Irving Kaplansky, and Alexander Schiemann.
\newblock There are 913 regular ternary forms.
\newblock {\em Mathematika}, 44(2):332--341, 1997.

\bibitem[Kan]{Kani:ys}
Earnst Kani.
\newblock Idoneal numbers and some generalizations.

\bibitem[Mag37]{Magnus:1937ly}
Wilhelm Magnus.
\newblock \"{U}ber die {A}nzahl der in einem {G}eschlecht enthaltenen {K}lassen
  von positiv-definiten quadratischen {F}ormen.
\newblock {\em Math. Ann.}, 114(1):465--475, 1937.

\bibitem[OS97]{Ono:1997nz}
Ken Ono and K.~Soundararajan.
\newblock Ramanujan's ternary quadratic form.
\newblock {\em Invent. Math.}, 130(3):415--454, 1997.

\bibitem[Pfe71]{Pfeuffer:1971pd}
Horst Pfeuffer.
\newblock Einklassige {G}eschlechter totalpositiver quadratischer {F}ormen in
  totalreellen algebraischen {Z}ahlk\"orpern.
\newblock {\em J. Number Theory}, 3:371--411, 1971.

\bibitem[Pfe79]{Pfeuffer:1979bh}
Horst Pfeuffer.
\newblock On a conjecture about class numbers of totally positive quadratic
  forms in totally real algebraic number fields.
\newblock {\em J. Number Theory}, 11(2):188--196, 1979.

\bibitem[Piz73]{Pizer:1973vn}
Arnold~K. Pizer.
\newblock Type numbers of {E}ichler orders.
\newblock {\em J. Reine Angew. Math.}, 264:76--102, 1973.

\bibitem[S{\etalchar{+}}11]{sage-4.6.2}
W.\thinspace{}A. Stein et~al.
\newblock {\em {S}age {M}athematics {S}oftware ({V}ersion 4.6.2)}.
\newblock The Sage~Development Team, 2011.
\newblock {\tt http://www.sagemath.org}.

\bibitem[Shi99]{Shimura:1999ad}
Goro Shimura.
\newblock An exact mass formula for orthogonal groups.
\newblock {\em Duke Math. J.}, 97(1):1--66, 1999.

\bibitem[Shi06]{Shimura:2006ac}
Goro Shimura.
\newblock Integer-valued quadratic forms and quadratic {D}iophantine equations.
\newblock {\em Doc. Math.}, 11:333--367 (electronic), 2006.

\bibitem[Shi10]{Shimura:2010uq}
Goro Shimura.
\newblock {\em Arithmetic of quadratic forms}.
\newblock Springer Monographs in Mathematics. Springer, New York, 2010.

\bibitem[Sie35]{Siegel:1935fk}
Carl~Ludwig Siegel.
\newblock {\"U}ber die classenzahl quadratischer zahlk{\"o}rper.
\newblock {\em Acta Arith.}, 1:83--86, 1935.

\bibitem[Voi]{Voigt:ys}
John Voigt.
\newblock Characterizing quaternion rings over an arbitrary base.
\newblock \texttt{http://www.cems.uvm.edu/~voight/}.

\bibitem[Wat62]{Watson:1962aa}
G.~L. Watson.
\newblock Transformations of a quadratic form which do not increase the
  class-number.
\newblock {\em Proc. London Math. Soc. (3)}, 12:577--587, 1962.

\bibitem[Wat84]{Watson:1984aa}
G.~L. Watson.
\newblock One-class genera of positive quadratic forms in seven variables.
\newblock {\em Proc. London Math. Soc. (3)}, 48(1):175--192, 1984.

\end{thebibliography}


\end{document}